\documentclass[11pt,reqno]{amsart}
\usepackage{mathrsfs}
\usepackage{amssymb}
\usepackage{amsfonts}
\usepackage{amsmath}
\usepackage{txfonts}
\usepackage{dsfont}
\usepackage{amscd}
\usepackage[colorlinks, linkcolor=blue, anchorcolor=blue, citecolor=blue]{hyperref}
\begin{document}
\setlength{\oddsidemargin}{0cm} \setlength{\evensidemargin}{0cm}
\baselineskip=20pt

\theoremstyle{plain} \makeatletter
\newtheorem{theorem}{Theorem}[section]
\newtheorem{Proposition}[theorem]{Proposition}
\newtheorem{Lemma}[theorem]{Lemma}
\newtheorem{Corollary}[theorem]{Corollary}

\theoremstyle{definition}
\newtheorem{notation}[theorem]{Notation}
\newtheorem{exam}[theorem]{Example}
\newtheorem{proposition}[theorem]{Proposition}
\newtheorem{conj}{Conjecture}
\newtheorem{prob}[theorem]{Problem}
\newtheorem{remark}[theorem]{Remark}
\newtheorem{claim}{Claim}
\newtheorem{Definition}[theorem]{Definition}
\newtheorem{Question}[theorem]{Question}

\newcommand{\SO}{{\mathrm S}{\mathrm O}}
\newcommand{\SU}{{\mathrm S}{\mathrm U}}
\newcommand{\Sp}{{\mathrm S}{\mathrm p}}
\newcommand{\so}{{\mathfrak s}{\mathfrak o}}
\newcommand{\Ad}{{\mathrm A}{\mathrm d}}
\newcommand{\m}{{\mathfrak m}}
\newcommand{\g}{{\mathfrak g}}
\newcommand{\h}{{\mathfrak h}}


\numberwithin{equation}{section}
\title[The moment map for the variety of  Leibniz algebras]{The moment map for the variety of  Leibniz algebras}
\author{Zhiqi Chen}
\address[Zhiqi Chen]{School of Mathematics and Statistics, Guangdong University of Technology, Guangzhou 510520, P.R. China}\email{chenzhiqi@nankai.edu.cn}
\author{Saiyu Wang}
\address [Saiyu Wang]{School of Mathematical Sciences and LPMC, Nankai University,
Tianjin 300071, P.R. China}\email{2120200040@mail.nankai.edu.cn}
\author{Hui Zhang}
\address [Hui Zhang]{School of Mathematics, Southeast University, Nanjing 210096, P.R. China}\email{2120160023@mail.nankai.edu.cn}

\begin{abstract}
We consider the moment map $m:\mathbb{P}V_n\rightarrow \textnormal{i}\mathfrak{u}(n)$ for the action of $\textnormal{GL}(n)$ on $V_n=\otimes^{2}(\mathbb{C}^{n})^{*}\otimes\mathbb{C}^{n}$, and   study  the functional  $F_n=\|m\|^{2}$   restricted to the projectivizations of the
algebraic varieties  of  all $n$-dimensional  Leibniz algebras $L_n$ and all $n$-dimensional symmetric Leibniz algebras $S_n$, respectively.
Firstly, we  give a description  of  the  maxima and minima of the functional  $F_n: L_n \rightarrow \mathbb{R}$,  proving that they  are actually  attained   at the symmetric Leibniz algebras. Then, for an arbitrary  critical point  $[\mu]$ of $F_n: S_n \rightarrow \mathbb{R}$,  we characterize  the structure of $[\mu]$ by  virtue of   the nonnegative rationality.  Finally,   we classify the critical  points of $F_n: S_n \rightarrow \mathbb{R}$  for $n=2$, $3$, respectively.
\end{abstract}
\keywords{Moment map; Variety of Leibniz algebras;   Critical point.}
\subjclass[2010]{14L30, 17B30, 53D20.}

\maketitle

\section{Introduction}
In \cite{Lauret03}, Lauret  studied  the moment map for the variety of Lie algebras    and   obtained  many remarkable results for example,
a stratification of the   Lie algebras variety and a description of the critical points,
which  turned to be  very  useful in proving that every Einstein solvmanifold is standard (\cite{Lauret2010}) and
in the characterization of solitons (\cite{BL2018,Lauret2011}). It is  thus  natural and interesting to ask whether  Lauret's results can be generalized, in some way,  to   varieties of   algebras beyond Lie algebras.

Motivated by the  idea, the study has recently  been extended to the  variety of   $3$-Lie algebras  in \cite{ZCL}. Here,  a  $3$-Lie algebra is
 a natural generalization
of the concept of a Lie algebra to the case where the fundamental multiplication operation
is $3$-ary.   See  also \cite{GKM} and  \cite{ZY} for the study of  the  moment map   in  Jordan and  associative algebras.

In this article,  we  study  the moment map for the variety of  \textit{Leibniz algebras}, which  are nonanticommutative versions of Lie algebras.
A Leibniz algebra is a vector space with a multiplication such that every left  multiplication operator is a derivation, which  was  at first  introduced by Bloh  (\cite{Bloh1965}) and later independently rediscovered  by Loday    in the study of  cohomology theory (see \cite{Loday1993,LodayPira1993}).     Leibniz algebras   play an important role in different areas of mathematics and physics \cite{BoHo2020,HoSa2019,KotStr,Lavau19,Lavau20,ShTaZh2021,Str2019,StrWa20},  and we refer to \cite{Feldvoss}  for a nice survey of Leibniz algebras.




For the moment map   in the frame of   Leibniz algebras,  it is defined    as follows: Let $\textnormal{GL}(n)$ be  the complex reductive  Lie group acting  naturally  on   the complex  vector  space $V_n=\otimes^{2}(\mathbb{C}^{n})^{*}\otimes\mathbb{C}^{n}$, i.e.,  the space of all $n$-dimensional complex algebras. The usual Hermitian inner product on $\mathbb{C}^{n}$   induces an $\textnormal{U}(n)$-invariant  Hermitian inner product on $V_n$, which is denoted by  $\langle\cdot,\cdot\rangle$.  Since $\mathfrak{gl}(n)=\mathfrak{u}(n)+\textnormal{i}\mathfrak{u}(n)$,  we may
define a function as follows
\begin{align*}
m: \mathbb{P} V_n \rightarrow \textnormal{i}\mathfrak{u}(n),\quad
(m([\mu]), A)=\frac{(\textnormal{d}\rho_\mu)_{e}A}{\|\mu\|^{2}}, \quad 0 \neq \mu \in V_n, ~~ A\in \textnormal{i}\mathfrak{u}(n),
\end{align*}
where $(\cdot,\cdot)$ is an $\textnormal{Ad}(\textnormal{U}(n))$-invariant real inner product on $\textnormal{i}\mathfrak{u}(n),$ and $\rho_\mu:\textnormal{GL}(n)\rightarrow\mathbb{R}$ is defined by $\rho_\mu(g)=\langle g.\mu, g.\mu\rangle$.
The function $m$ is the moment
map from symplectic geometry, corresponding to the Hamiltonian action  $\textnormal{U}(n)$ of $V_n$ on
the symplectic manifold $\mathbb{P}V_n$ (see \cite{Kirwan98,Ness1984}).
In this article, we  shall  study    the critical points of the functional   $F_n=\|m\|^{2}: \mathbb{P} V_n \rightarrow \mathbb{R}$, and emphasize  those  critical points that lie in   $L_n$ and $S_n$. Here,  $L_n,S_n$ denote  the projectivizations of the
algebraic varieties  of  all $n$-dimensional   Leibniz algebras, and  all $n$-dimensional symmetric Leibniz algebras, respectively.

The article is organized as follows: In Section~\ref{section2}, we recall  some fundamental results of Leibniz algebras (Def.~\ref{l-s}) and symmetric Leibniz algebras (Def.~\ref{sym}).

In Section~\ref{section3}, we  first give the explicit expression of the moment map $m:\mathbb{P} V_n \rightarrow \textnormal{i}\mathfrak{u}(n)$ in terms of $\textnormal{M}_{\mu}$, in fact $m([\mu])=\frac{\textnormal{M}_{\mu}}{\|\mu\|^{2}}$, $[\mu]\in \mathbb{P}V_n$ (Lemma~\ref{Mm}). Then we show that  $[\mu]\in \mathbb{P}V_n$  is   a critical point of $F_n=\|m\|^{2}: \mathbb{P} V_n \rightarrow \mathbb{R}$ if and only if    $\textnormal{M}_{\mu}=c_{\mu} I+D_{\mu}$ for some $c_{\mu} \in \mathbb{R}$ and $D_{\mu} \in \textnormal{Der}(\mu)$ (Thm.~\ref{MID}).

In  Section~\ref{section4},     we prove that  there exists a constant $c>0$ such that the eigenvalues of $cD_\mu$ are integers   for  any critical point $[\mu]\in \mathbb{P}V_n$,  and  if moreover  $[\mu]\in S_n$, we show that the eigenvalues are necessarily   nonnegative (Thm.~\ref{eigenvalue}), which  generalizes the nonnegative  rationality from Lie algebras to symmetric Leibniz algerbas (see   \cite[Thm 3.5]{Lauret03}).
Besides,  we  give a description of  the  extremal points of $F_n:L_n \rightarrow \mathbb{R},$  proving that the minimum value is
attained at semisimple Lie algebras (Thm.~\ref{x}),  while the maximum value is attained  at the direct sum of the  two-dimensional non-Lie symmetric Leibniz algebra with the trivial algebra (Thm.~\ref{max}). Finally, for an arbitrary  critical point  $[\mu]$ of $F_n: S_n \rightarrow \mathbb{R}$,  we characterize the structure of $[\mu]$ by  virtue of   the nonnegative rationality  of $D_\mu$ (Thm.~\ref{structure}--Thm.~\ref{converse}).

In Section~\ref{examples}, we classify the critical points of $F_n: S_n \rightarrow \mathbb{R}$ with $n=2,3,$ which shows that  there exist many critical points that are not Lie algebras. Moreover, we prove that every 2-dimensional   symmetric Leibniz algebra is isomorphic to a critical point
of $F_2;$ and there exist    3-dimensional   symmetric Leibniz algebras  which are not isomorphic to  any  critical point of  $F_3$.

Finally  in Section~\ref{ques}, we  summary the article and collect some natural questions concerning the critical points of $F_n: L_n \rightarrow \mathbb{R}$.

\section{Preliminaries}\label{section2}
In this section,  we recall some basic definitions and results of Leibniz algebras . The ambient field is always assumed to be the complex number field $\mathbb{C}$ unless otherwise  stated.

\begin{Definition}[\cite{Feldvoss,Loday1993}]\label{l-s}
A vector space $\mathfrak{l}$ over  $\mathbb{C}$ with a bilinear map
$\mathfrak{l}\times \mathfrak{l}\rightarrow \mathfrak{l}$, denoted by $(x,y)\mapsto xy$, is called a \textit{Leibniz algebra}, if every left multiplication is a derivation, i.e.,
\begin{align}\label{left}
x(yz)=(xy)z+y(xz)
\end{align}
for all $x,y,z\in \mathfrak{l}.$
\end{Definition}
\begin{remark}
Leibniz algebras are sometimes called \textit{left} Leibniz algebras in the literature, and there is a corresponding notion of \textit{right} Leibniz algebra,
i.e., an algebra  with the property that  every right multiplication is a derivation. In some studies, the authors prefer to call  a right Leibniz algebra a Leibniz algebra. We point out that for our purpose, it actually does not matter which notion is used since the opposite algebra of a left Leibniz algebra is a right Leibniz algebra and vice versa.
\end{remark}

Following Mason and Yamskulna \cite{MY2013},  we introduce the notion of  the symmetric Leibniz algebra as follows.
\begin{Definition}[\cite{MY2013}]\label{sym}
An algebra $\mathfrak{l}$ is called a  \textit{symmetric Leibniz algebra} if it is at the same time   a left  and  a right Leibniz algebra, that is
\begin{align}\label{l-r}
x(yz)&=(xy)z+y(xz), \\
(xy)z&=(xz)y+x(yz),
\end{align}
for all $x,y,z\in \mathfrak{l}.$
\end{Definition}

Every Lie algebra is clearly  a  symmetric Leibniz algebra, and the converse is not true.
In the following,    we make the convention that  an ideal of a Leibniz algebra  always means a two-side ideal.


\begin{Definition}\label{solvable}
Let $\mathfrak{l}$ be a   Leibniz algebra. $\mathfrak{l}$ is called solvable if $\mathfrak{l}^{(r)}=0$ for some $r\in \mathbb{N}$, where $\mathfrak{l}^{(0)}=\mathfrak{l}, \mathfrak{l}^{(k+1)}=\mathfrak{l}^{(k)}\mathfrak{l}^{(k)},k\geq0.$
\end{Definition}

If $I,J$ are any two solvable ideals of $\mathfrak{l}$, then $I+J$ is also  a solvable ideal of $\mathfrak{l}$, so the maximum solvable ideal is unique, called the $radical$ of $\g$ and denoted by $\textnormal{Rad}(\mathfrak{l})$ (\cite{Feldvoss}).
\begin{theorem}[\cite{Barnes2021}]\label{levi}
A Leibniz algebra $\mathfrak{l}$ over a field of characteristic $0$ admits a Levi decomposition, i.e., $\mathfrak{l}=\mathcal{S}+\textnormal{Rad}(\mathfrak{l})$ decomposes into the sum of a semisimple Lie subalgebra $\mathcal{S}$   and the radical satisfying  $\mathcal{S}\cap\textnormal{Rad}(\mathfrak{l})=0.$
\end{theorem}

\begin{Definition}\label{nilpotent}
A   Leibniz algebra $\mathfrak{l}$ is called \textit{nilpotent} if there exists a positive integer $n$ such that any
product of $n$ elements in $\mathfrak{l}$, no matter how associated, is zero.
\end{Definition}

For a  Leibniz algebra, we define
$^{1}\mathfrak{l}:=\mathfrak{l}, ~ ^{k+1}\mathfrak{l}:=\mathfrak{l}(^{k}\mathfrak{l}), k\geq 1.$
Furthermore,  we define
\begin{align*}
\mathfrak{l}_1:=\mathfrak{l},\quad \mathfrak{l}_{k}=\sum_{i=1}^{k-1}\mathfrak{l}_i\mathfrak{l}_{k-i},~~ k\geq 2.
\end{align*}
Then we have the following theorem.
\begin{theorem}[\cite{Feldvoss}]\label{nil}
For any integer $k\geq 1$,  then $^{k}\mathfrak{l}=\mathfrak{l}_k$.
Moreover, $\mathfrak{l}$ is nilpotent if and only if there exists an positive integer $n$ such that $\mathfrak{l}_n=0.$
\end{theorem}
If $I,J$ are  two nilpotent ideals of a Leibniz algebra $\mathfrak{l}$, then $I+J$ is also   a nilpotent ideal of $\mathfrak{l}$, consequently  the maximum nilpotent ideal is unique, called the \textit{nilradical}, denoted by $\textnormal{N}(\mathfrak{l})$ (\cite{Feldvoss,Towers2021}).

\begin{Proposition}[\cite{Towers2021}]
Let $\mathfrak{l}$ be a Leibniz algebra over a field of characteristic zero, then $\textnormal{Rad}(\mathfrak{l})\mathfrak{l}$, $\mathfrak{l}\textnormal{Rad}(\mathfrak{l})\subset\textnormal{N}(\mathfrak{l})$.

\end{Proposition}

\section{The moment map for  complex algebras}\label{section3}
In this section, we first recall  Lauret's idea: \textit{varying  brackets instead of  metrics}, for the study of metric algebras, then we introduce the moment map for complex algebras.

Let $\mathbb{C}^{n}$ be  the $n$-dimensional  complex vector space.  A metric algebra is  a triple  $(\mathbb{C}^n,\mu,\langle  \cdot  , \cdot \rangle)$, where $\mu:\mathbb{C}^{n}\times\mathbb{C}^{n}\rightarrow\mathbb{C}^n$ is a bilinear map and $\langle  \cdot  , \cdot \rangle$ is a Hermitian inner product on $\mathbb{C}^n$. The triple $(\mathbb{C}^n,\mu,\langle  \cdot  , \cdot \rangle)$  will be  abbreviated as $(\mu,\langle  \cdot  , \cdot \rangle)$ in this article.
\begin{Definition}\label{isometryAndIsometryUpToScaling}
Let $(\mu_{1},\langle  \cdot  , \cdot \rangle _{1})$ and $(\mu_{2},\langle  \cdot  , \cdot \rangle _{2})$ be two metric  algebras.
\begin{enumerate}
\item They are said to be isomorphic if there exists linear isomorphism $\varphi: \mathbb{C}^n\to \mathbb{C}^n$ such that $\varphi(\mu_1(\cdot  , \cdot))=\mu_2 (\varphi(\cdot)  , \varphi(\cdot)),$ and in this case, $\varphi$ is called an algebra isomorphism.
\item They are said to be isometric if there exists an algebra isomorphism $\varphi$ such that $\langle  \cdot  , \cdot \rangle _{1}=\langle \varphi(\cdot)  , \varphi(\cdot) \rangle _{2}.$
\item They are said to be isometric up to scaling if there exists an algebra isomorphism $\varphi$ and $c>0$ such that $\langle  \cdot  , \cdot \rangle _{1}=c\langle \varphi(\cdot)  , \varphi(\cdot) \rangle _{2}.$
\end{enumerate}
\end{Definition}
\begin{remark}
The   Definition~\ref{isometryAndIsometryUpToScaling}   is an analogy of \cite{HTT,M76}, where the (real) metric    Lie algebras and their relations with Riemannian geometry, such as sectional curvatures, left-invariant Einstein metrics and Ricci solitons, are studied.
\end{remark}

Let $V_n=\otimes^{2}({\mathbb{C}^{n}})^{*}\otimes\mathbb{C}^{n}$ be the space of all bilinear maps, and
$$
\mathfrak{M}_n=\{\langle\cdot,\cdot\rangle :  \langle\cdot,\cdot\rangle \textnormal{~is a Hermitian inner product on } \mathbb{C}^{n}\}.
$$
be the moduli space of all  Hermitian inner products on $\mathbb{C}^{n}$, respectively.
Consider the  natural action  of $\textnormal{GL}(n)=\textnormal{GL}(\mathbb{C}^{n})$ on  $V_n$,  i.e.,
\begin{align}\label{Gaction}
g.\mu(X,Y)=g\mu(g^{-1}X,g^{-1}Y),\quad g\in \textnormal{GL}(n), X,Y\in\mathbb{C}^{n}.
\end{align}
then  by Definition~\ref{isometryAndIsometryUpToScaling}, we know that $\textnormal{GL}(n).\mu$ is precisely the isomorphism class of $\mu$. Moreover,  differentiating  (\ref{Gaction}), we obtain the natural action $\mathfrak{gl}(n)$ on  $V_n:$
\begin{align}\label{gaction}
A.\mu(X,Y)=A\mu(X,Y)-\mu(AX,Y)-\mu(X,AY), \quad A\in \mathfrak{gl}(n),\mu\in V_n.
\end{align}
It follows that $A.\mu=0$ if and only if $A\in\textnormal{Der}(\mu),$ the derivation  algebra of $\mu.$
On the other hand,  one knows that the linear group  $\textnormal{GL}(n)$ also acts   on $\mathfrak{M}_n$, i.e.,
\begin{align*}
g.\langle\cdot,\cdot\rangle=\langle g^{-1}(\cdot),g^{-1}(\cdot)\rangle, \quad g\in \textnormal{GL}(n),
\end{align*}
and  this action is  obviously  transitive.

\begin{Lemma}\label{iso}
Two metric  algebras $(\mu, \langle\cdot,\cdot\rangle_1)$ and $(\lambda, \langle\cdot,\cdot\rangle_2)$ are   isometric up to scaling if and only if there exist $g\in \textnormal{GL}(n)$ and $c\ne0$ such that $\lambda=g.\mu$ and $\langle\cdot,\cdot\rangle_2=(cg).\langle\cdot,\cdot\rangle_1.$ In particular, $((cg)^{-1}.\mu, \langle\cdot,\cdot\rangle)$ and $(\mu, g.\langle\cdot,\cdot\rangle)$ are isometric up to scaling.
\end{Lemma}

Fix a Hermitian inner product $\langle\cdot,\cdot\rangle$ on $\mathbb{C}^{n}$,  then by Lemma~\ref{iso} we have
$$\bigcup_{g\in \textnormal{GL}(n)}(g.\mu,\langle\cdot,\cdot\rangle)=\bigcup_{g\in \textnormal{GL}(n)}(\mu, g^{-1}.\langle\cdot,\cdot\rangle)$$
in the sense of isometry (Definition~\ref{isometryAndIsometryUpToScaling}). This is precisely the idea: \textit{varying brackets instead of  metrics}, for the study of  metric algebras. By this idea, Lauret  introduced   the moment map for Lie algebras, which has motivated much of the recent study of  homogeneous Riemannian geometry \cite{BL2018,BL2023,lauret2001,Lauret2010,Lauret2011}.

Now, we introduce the moment map for complex algebras. Fix a  Hermitian inner product $\langle\cdot,\cdot\rangle$ on $\mathbb{C}^{n}$, then  it  makes each $\mu\in V_n$  an metric algebra,   and $\textnormal{U}(n).\mu$ is precisely the isometry class of $\mu$ (see Definition~\ref{isometryAndIsometryUpToScaling}).    Moreover,  $\langle\cdot,\cdot\rangle$   induces  a natural  $\textnormal{U}(n)$-invariant  Hermitian inner product on $V_n$ as follows
\begin{align}\label{metric}
\langle\mu,\lambda\rangle:=\sum_{i,j,k}\langle\mu(X_i,X_j),X_{k}\rangle\overline{\langle\lambda(X_i,X_j),X_{k}\rangle},\quad\quad~\mu,\lambda\in  V_n,
\end{align}
where $\{X_1,X_2,\cdots, X_n\}$ is an arbitrary orthonormal basis of  $(\mathbb{C}^{n},\langle\cdot,\cdot\rangle)$. Note that  there is an  $\textnormal{Ad}(\textnormal{U}(n))$-invariant Hermitian inner product on $\mathfrak{gl}(n)$,  i.e.,
\begin{align}\label{product}
(A,B)=\operatorname{tr}AB^{*},~ A,B\in \mathfrak{gl}(n).
\end{align}
where $^{*}$ denotes the conjugate  transpose relative to $(\mathbb{C}^{n},\langle\cdot,\cdot\rangle)$.
The moment map,   corresponding to the Hamiltonian action  of $\textnormal{U}(n)$ on the symplectic
manifold  $\mathbb{P}V_n$,  is defined by
\begin{align}\label{moment}
m: \mathbb{P} V_n \rightarrow \textnormal{i}\mathfrak{u}(n), \quad(m([\mu]), A)=\frac{(\textnormal{d}\rho_\mu)_{e}A}{\|\mu\|^{2}}, \quad 0 \neq \mu \in V_n, A\in \textnormal{i}\mathfrak{u}(n),
\end{align}
where  $\rho_\mu(g)=\langle g.\mu, g.\mu\rangle$, $g\in \textnormal{GL}(n)$. Denote by
$F_n: \mathbb{P} V_n \rightarrow \mathbb{R},  F_n([\mu])=\|m([\mu])\|^{2}=(m([\mu]),m([\mu])),$ the square norm of the moment map. Then it is easy to see that the moment map is $\textnormal{U}(n)$-invariant, i.e.,  $m(k.[\mu])=\textnormal{Ad}(k)m([\mu])$, $\forall k\in \textnormal{U}(n)$. In particular,   $F_n(k.[\mu])=F_n([\mu])$, $\forall k\in \textnormal{U}(n)$.

For each  algebra $\mu\in V_n$, we define $\textnormal{M}_{\mu}\in\textnormal{i}\mathfrak{u}(n)$  as follows
\begin{align}\label{M}
\textnormal{M}_{\mu}=2\sum_{i}{L}^{\mu}_{X_i}({L}^{\mu}_{X_i})^{*}-2\sum_{i}({L}^{\mu}_{X_i})^{*}{L}^{\mu}_{X_i}-2\sum_{i}({R}^{\mu}_{X_i})^{*}{R}^{\mu}_{X_i},
\end{align}
where    ${L}^{\mu}_{X},{R}^{\mu}_{X}: \mathbb{C}^{n} \rightarrow \mathbb{C}^{n}$ are given  by ${L}^{\mu}_{X}(Y)=\mu(X,Y)$ and ${R}^{\mu}_{X}(Y)=\mu(Y,X)$, $\forall Y\in\mathbb{C}^{n}$, and $^{*}$ denotes the conjugate  transpose relative to $(\mathbb{C}^{n},\langle\cdot,\cdot\rangle)$.  One immediately   sees that  $\textnormal{M}_{k.\mu}=\textnormal{Ad}(k)\textnormal{M}_{\mu}$ for any $k\in \textnormal{U}(n),$ and $\textnormal{M}_{c\mu}=|c|^2\textnormal{M}_{\mu}$ for any $0\ne c\in \mathbb{C}.$
Moreover
\begin{align}\label{Mformula}
\langle\textnormal{M}_{\mu} X, Y\rangle=&2 \sum_{i,j} \overline{\langle\mu( X_i,X_j), X\rangle}\langle\mu(X_i,X_j), Y\rangle-2 \sum_{i,j}\langle\mu(X_i,X), X_{j}\rangle \overline{\langle\mu(X_i,Y), X_{j}\rangle} \notag\\
&-2 \sum_{i,j}\langle\mu(X, X_i), X_{j}\rangle \overline{\langle\mu(Y, X_i), X_{j}\rangle}
\end{align}
for any $X,Y\in\mathbb{C}^{n}.$  Note that if  the algebra $\mu$ is  anticommutative, then  $\textnormal{M}_{\mu}$ coincides with  \cite{Lauret03}.

The next lemma  establishes  the   relation between $m([\mu])$ and $\textnormal{M}_{\mu}$, which follows from a straightforward calculation (see also \cite{ZY}).
\begin{Lemma}\label{Mm}
For  any  $0\ne\mu\in V_n$, we have $m([\mu])=\frac{\textnormal{M}_{\mu}}{\|\mu\|^{2}}.$
In particular,  $(\textnormal{M}_{\mu}, A)=2\langle A.\mu,\mu\rangle$ for any $A\in \textnormal{i}\mathfrak{u}(n).$
\end{Lemma}

\begin{Corollary}\label{MD}
For  any  $\mu\in V_n$, then
\begin{enumerate}
\item [(i)] $\operatorname{tr}\textnormal{M}_{\mu}D=0$ for any $D\in \textnormal{Der}(\mu)\cap \textnormal{i}\mathfrak{u}(n);$
\item [(ii)] $\operatorname{tr}\textnormal{M}_{\mu}[A,A^{*}]\geq0$ for any $A\in \textnormal{Der}(\mu),$ and equality holds if and only if $A^{*}\in\textnormal{Der}(\mu).$
\end{enumerate}
\end{Corollary}
\begin{proof}
For \textnormal{(i)}, it follows from Lemma~\ref{Mm} and the fact that $D$ is a Hermitian derivation of $\mu.$ For \textnormal{(ii)}, it follows  from that $\operatorname{tr}\textnormal{M}_{\mu}[A,A^{*}]=2\langle A^{*}.\mu,A^{*}.\mu\rangle\geq0$ for any $A\in \textnormal{Der}(\mu)$, and the fact  $A^{*}.\mu=0$ if and only if $A^{*}\in\textnormal{Der}(\mu).$
\end{proof}

\begin{theorem}[\cite{ZY}]\label{MID}
For the  square norm of the moment map $F_{n}=\|m\|^{2}: \mathbb{P} V_{n} \rightarrow \mathbb{R}$,  the following statements are equivalent:
\begin{enumerate}
\item [(1)] $[\mu] \in \mathbb{P} V_{n}$ is a critical point of $F_{n}$.
\item [(2)] $[\mu] \in \mathbb{P} V_{n}$ is a critical point of $F_{n}|_{\textnormal{GL}(n).[\mu]}.$
\item [(3)] $\textnormal{M}_{\mu}=c_{\mu} I+D_{\mu}$ for some $c_{\mu} \in \mathbb{R}$ and $D_{\mu} \in \textnormal{Der}(\mu)$.
\end{enumerate}
If one of the  statements holds, then
\begin{enumerate}
\item [(i)] $c_\mu=\frac{\operatorname{tr}\textnormal{M}_{\mu}^{2}}{\operatorname{tr}\textnormal{M}_{\mu}}=-\frac{1}{2}\frac{\operatorname{tr}\textnormal{M}_{\mu}^{2}}{\|\mu\|^{2}}<0.$
\item [(ii)] If $\operatorname{tr}D_\mu\ne0$, then $c_\mu=-\frac{\operatorname{tr}{D}_{\mu}^{2}}{\operatorname{tr}{D}_{\mu}}$ and $\operatorname{tr}D_\mu>0.$
\end{enumerate}
\end{theorem}

\begin{remark}\label{measures}
By Lemma~\ref{M}, we know that   $\operatorname{tr}\textnormal{M}_{\mu}=-2\langle\mu , \mu\rangle=-2\|\mu\|^{2}$ for all $0\ne \mu\in V_n.$ On the other hand $$\|m[\mu]-\frac{\operatorname{tr}{m[\mu]}}{n}I\|^{2}=\|m[\mu]\|^{2}-2\cdot\frac{\operatorname{tr}{m[\mu]}}{n}\cdot\operatorname{tr}{m[\mu]}
+\left(\frac{\operatorname{tr}{m[\mu]}}{n}\right)^2\cdot n.$$ So by Lemma~\ref{M},   we have
\begin{align*}
\|m[\mu]-\frac{\operatorname{tr}{m[\mu]}}{n}I\|^{2}=F_{n}([\mu])-\frac{4}{n}.
\end{align*}
That is,  $F_{n}([\mu])$ measures in some sense  how far $m([\mu])$ is from the identity. If we  interpret  $\textnormal{M}_{\mu}$ as some 'curvature' of the metric algebra $(\mu,\langle  \cdot  , \cdot \rangle)$, which is of course invariant under isometry (Definition~\ref{isometryAndIsometryUpToScaling}),
then the critical point of $F_n$  can  be thought to find the 'best' Hermitian inner product (realtive to the curvature) in an isomorphism class of  an algebra.
\end{remark}

Moreover, note that for any $\mu \in V_n,$  $0$ lies in the boundary of $\textnormal{GL}(n).\mu,$   so a result due to Ness can be stated as follows
\begin{theorem}[\cite{Ness1984}]\label{min_cri} If $[\mu]$ is a critical point of the functional $F_{n}: \mathbb{P} V_{n} \mapsto \mathbb{R}$  then
\begin{enumerate}
\item [(i)] $\left.F_{n}\right|_{\mathrm{GL} (n) .[\mu]}$ attains its minimum value at $[\mu]$.
\item [(ii)] $[\lambda] \in \mathrm{GL} (n).[\mu]$ is a critical point of $F_{n}$ if and only if $[\lambda] \in \mathrm{U}(n).[\mu]$.
\end{enumerate}
\end{theorem}
The Ness theorem  says that if $[\mu]\in \mathbb{P} V_{n}$ is a critical point, then $\mathrm{U} (n).[\mu]$ is precisely the set of critical points that are contained in $\mathrm{GL} (n).[\mu].$  Those  $\mathrm{GL} (n)$-orbits, which contain a critical point, are called distinguished orbits in the literature.

\section{The critical points of the variety of   Leibniz  algebras }\label{section4}
The spaces $\mathscr{L}_n$,  $\mathscr{S}_n$ of all $n$-dimensional  Leibniz algebras and symmetric Leibniz algebras   are  algebraic sets since they are given by polynomial conditions.
Denote by $L_n$ and  $S_n$  the projective algebraic varieties obtained by projectivization of $\mathscr{L}_n$ and  $\mathscr{S}_n$, respectively. Then by Theorem~\ref{MID}, we know that the critical points of $F_{n}: L_{n} \rightarrow \mathbb{R}$, and $F_{n}: S_{n} \rightarrow \mathbb{R}$ are precisely the critical points
of $F_n:\mathbb{P}V_n\rightarrow\mathbb{R}$ which lie in $L_n$ and $S_n$, respectively.
\subsection{The rationality and nonnegative property}
The following    rationality  and nonnegative property  are  generalizations of  \cite{Lauret03} from Lie algebras to Leibniz  algebras and symmetric Leibniz  algebras, respectively.

\begin{theorem}\label{eigenvalue}
Let $[\mu]\in \mathbb{P}V_n$ be a critical point of $F_n:\mathbb{P}V_n\rightarrow\mathbb{R}$ with $\textnormal{M}_{\mu}=c_{\mu} I+D_{\mu}$ for some $c_{\mu} \in \mathbb{R}$ and $D_{\mu} \in \textnormal{Der}(\mu)$. Then there exists a constant $c>0$ such that the eigenvalues of $cD_\mu$ are  integers  prime to each other, say $k_1< k_2<\cdots < k_r \in \mathbb{Z}$ with multiplicities $d_1,d_2,\cdots,d_r\in \mathbb{N}.$ If moreover $[\mu]\in S_{n}$, then  the integers are nonnegative.
\end{theorem}
\begin{proof}
The first part follows from    \cite{ZCL} (see also \cite{ Lauret03}). We only prove the last statement.
The case $D_\mu=0$ is trivial.  In the sequel, we assume that $D_\mu$ is nonzero. Noting that  $D_\mu$ is  Hermitian,    there is an  orthogonal decomposition
\begin{align*}
\mathbb{C}^{n}=\mathfrak{l}_1\oplus\mathfrak{l}_2\oplus\cdots\oplus\mathfrak{l}_r,~~ r\geq 2
\end{align*}
where $\mathfrak{l}_i:=\{X\in\mathbb{C}^{n}|D_\mu X=c_i X\}$ are the eigenspaces of $D_\mu$ corresponding to the eigenvalues $c_1<c_2<\cdots<c_r\in\mathbb{R}$, respectively. Suppose  that $[\mu]\in S_{n}$, and  $0\ne X\in\mathbb{C}^{n}$ satisfies $D_\mu X=c_1X$. Then we have
\begin{align*}
c_1 L_X^{\mu}=[D_\mu,L_X^{\mu}],\\
c_1 R_X^{\mu}=[D_\mu,R_X^{\mu}].
\end{align*}
It follows that
\begin{align}\label{c1-l}
c_1\operatorname{tr} L_X^{\mu}(L_X^{\mu})^{*}=\operatorname{tr}[D_\mu,L_X^{\mu}](L_X^{\mu})^{*}=\operatorname{tr}[\textnormal{M}_\mu,L_X^{\mu}](L_X^{\mu})^{*}
=\operatorname{tr}\textnormal{M}_\mu[L_X^{\mu},(L_X^{\mu})^{*}],
\end{align}
and
\begin{align}\label{c1-r}
c_1\operatorname{tr} R_X^{\mu}(R_X^{\mu})^{*}=\operatorname{tr}[D_\mu,R_X^{\mu}](R_X^{\mu})^{*}=\operatorname{tr}[\textnormal{M}_\mu,R_X^{\mu}](R_X^{\mu})^{*}
=\operatorname{tr}\textnormal{M}_\mu[R_X^{\mu},(R_X^{\mu})^{*}].
\end{align}
Since $L_X^{\mu},R_X^{\mu}$ are derivations of $\mu,$   we conclude from  Corollary~\ref{MD} that
\begin{align*}
 c_1\operatorname{tr} L_X^{\mu}(L_X^{\mu})^{*}\geq 0 \quad \textnormal{and}\quad c_1\operatorname{tr} R_X^{\mu}(R_X^{\mu})^{*}\geq 0.
\end{align*}
If  $L_X^{\mu}$ or $R_X^{\mu}$ is not zero, then $c_1\geq 0.$ If  $L_X^{\mu}$ and $R_X^{\mu}$ are both zero, then $X$ necessarily lies in the center of $\mu$. By (\ref{Mformula}), we have
\begin{align}
\langle\textnormal{M}_{\mu} X, X\rangle=2\sum_{i,j}|\langle\mu(X_i,X_j), X\rangle|^{2}\geq 0.
\end{align}
Since $\textnormal{M}_{\mu}=c_\mu I+D_\mu$,  then $0\leq \langle\textnormal{M}_{\mu} X, X\rangle=(c_\mu+c_1)\langle X, X\rangle$. Using Theorem~\ref{MID},  we know $c_1\geq-c_\mu>0.$
This completes the proof.
\end{proof}

\begin{theorem}\label{positive}
Let $[\mu]$ be a critical point of $F_{n}: S_{n} \rightarrow \mathbb{R}$ with $\textnormal{M}_{\mu}=c_{\mu} I+D_{\mu}$ for some $c_{\mu} \in \mathbb{R}$ and $D_{\mu} \in \textnormal{Der}(\mu)$.   If   $[\mu]$ is  nilpotent, then  $D_\mu$ is positive definite. Consequently, all nilpotent critical points of $F_{n}: S_{n} \rightarrow \mathbb{R}$ are $\mathbb{N}$-graded.
\end{theorem}
\begin{proof}
 Indeed, assume that $0\ne X\in\mathbb{C}^{n}$ satisfies $D_\mu X=c_1X$, where $c_1$ is the smallest eigenvalue of $D_\mu$.  By Theorem~\ref{eigenvalue}, we know that $c_1\geq 0.$ Suppose that $c_1=0,$ then $\operatorname{tr}\textnormal{M}_\mu[L_X^{\mu},(L_X^{\mu})^{*}]=0,$ and $\operatorname{tr}\textnormal{M}_\mu[R_X^{\mu},(R_X^{\mu})^{*}]=0.$ Using Corollary~\ref{MD},  $(L_X^{\mu})^{*}$ and $(R_X^{\mu})^{*}$ are derivations of $\mu.$ Let $\mathfrak{l}$ be the symmetric Leibniz algebra $(\mathbb{C}^{n},\mu)$.
Consider the orthogonal decomposition of $\mathfrak{l}$
\begin{align*}
\mathfrak{l}=\mathfrak{n}_1\oplus\mathfrak{n}_2\oplus\cdots\oplus\mathfrak{n}_p,~~ p\geq 2,
\end{align*}
where  $\mu (\mathfrak{l},\mathfrak{l})=\mathfrak{n}_2\oplus\cdots\oplus\mathfrak{n}_p,$ $\mu(\mathfrak{l},\mu (\mathfrak{l},\mathfrak{l}))=\mathfrak{l}_3\oplus\cdots\oplus\mathfrak{l}_p,\cdots.$ Since  $(L_X^{\mu})^{*}$ is  a derivation of $\mu,$ then $(L_X^{\mu})^{*}$ necessarily leaves each $\mathfrak{l}_i$ invariant. Note that $L_X^{\mu}(\mathfrak{l}_i)\subset\mathfrak{l}_{i+1}$ for each $i$, then
 $\operatorname{tr} L_X^{\mu}(L_X^{\mu})^{*}=0$,  and consequently,  $L_X^{\mu}=0.$ Similarly, one concludes that $R_X^{\mu}=0.$ That is, $X$ lies in the center of $\mathfrak{l}$, which is a contradiction since in this case  we have $c_1\geq-c_\mu>0.$ So $D_\mu$ is positive definite.
\end{proof}

The positive argument in Theorem~\ref{positive} for real nilpotent Lie algebras  plays a fundamental  role in \cite{Lauret2010}.
\begin{remark}\label{n-e}
So far, it is still unclear for us  whether  the nonnegative  property  in Theorem~\ref{eigenvalue} and  positive property in Theorem~\ref{positive}   hold for  $F_{n}: L_{n} \rightarrow \mathbb{R}$ or not.  However, we have the following partial result. For an arbitrary critical point $[\mu]$ of $F_{n}: L_{n} \rightarrow \mathbb{R}$, consider
\begin{align*}
\mathfrak{l}=\mathfrak{l}_{-}\oplus\mathfrak{l}_0\oplus\mathfrak{l}_{+},
\end{align*}
 the direct sum of eigenspaces of $D_\mu$  with eigenvalue smaller than zero,  equal to zero and larger than zero, respectively. Then $R_X^{\mu}\notin \textnormal{Der}(\mu)$ for any $0\ne X\in \mathfrak{l}_{-},$ which in turn is equivalent to $[R_X^{\mu},(R_X^{\mu})^{*}]\ne0$, i.e., not normal.  Indeed, assume that $R_X^{\mu}\in \textnormal{Der}(\mu)$  (or $[R_X^{\mu},(R_X^{\mu})^{*}]=0$) for some $0\ne X\in \mathfrak{l}_{-}$, then  by the proof of Theorem~\ref{eigenvalue}, we see that $X$ necessarily lies in the center of $\mu$, which  contradicts $0\ne X\in \mathfrak{l}_{-}$.
\end{remark}

\subsection{The  minima and  maxima  of $F_{n}: L_{n} \rightarrow \mathbb{R}$} Following  from \cite{Lauret03}, we   introduce  the notion of the type of a critical point.
\begin{Definition}
The data set $(k_1<k_2<\cdots<k_r;d_1,d_2,\cdots,d_r)$ in \textnormal{Theorem}~\textnormal{\ref{eigenvalue}}  is called the
type of the critical point $[\mu].$
\end{Definition}

To study the  minima and  maxima  of $F_{n}: L_{n} \rightarrow \mathbb{R}$,  we recall two simple but useful results as follows.

\begin{Lemma}[\cite{ZCL}]\label{CV}
Let $[\mu]\in\mathbb{P}V_n$ be a critical point of $F_n$ with type $\alpha=(k_1<k_2<\cdots<k_r;d_1,d_2,\cdots,d_r).$ Then we have
\begin{enumerate}
\item [(i)] If $\alpha=(0;n)$, then  $F_n ([\mu])=\frac{4}{n}.$
\item [(ii)] If $\alpha\ne(0;n)$, then  $F_n([\mu])=4\left(n-\frac{(k_1d_1+k_2d_2+\cdots+k_rd_r)^{2}}{(k_1^{2}d_1+k_2^{2}d_2+\cdots+k_r^{2}d_r)}\right)^{-1}.$
\end{enumerate}
\end{Lemma}

\begin{Lemma}\label{4/n}
Assume $[\mu]\in\mathbb{P}V_n,$ then $[\mu]$ is a critical point of $F_n:\mathbb{P}V_n\rightarrow\mathbb{R}$ with type $(0;n)$ if and only if $F_n([\mu])=\frac{4}{n}.$ Moreover, $\frac{4}{n}$ is the  minimum  value of $F_n:\mathbb{P}V_n\rightarrow\mathbb{R}.$
\end{Lemma}
\begin{proof}
Using Remark~\ref{measures}
\end{proof}

The following theorem shows that  even in the frame of Leibniz algebras,  the semisimple Lie algebras are still the only  critical points of $F_{n}: L_{n} \rightarrow \mathbb{R}$ attaining  the minimum value.
\begin{theorem}\label{x}
Assume that there exists a  semisimple Lie algebra of dimension $n$. Then $F_{n}: L_{n} \rightarrow \mathbb{R}$ attains its minimum value at a point $[\lambda] \in \textnormal{GL}(n).[\mu]$ if and only if $\mu$ is a semisimple Lie algebra. In such a case, $F_{n}([\lambda])=\frac{4}{n}.$
\end{theorem}
\begin{proof}
Assume that  $\mu$ is  a complex semisimple Lie algebra,  then it follows from \cite[Theorem 4.3]{Lauret03} that $F_{n}([\lambda])=\frac{4}{n}$ for some  $[\lambda] \in \textnormal{GL}(n).[\mu]$.

Conversely,  assume $F_{n}: L_{n} \rightarrow \mathbb{R}$ attains its minimum value at a point $[\lambda] \in \textnormal{GL}(n).[\mu].$  Then by hypothesis, there exists a semisimple Lie algebra of dimension $n$. The first part of the proof and Lemma~\ref{4/n} imply that $\textnormal{M}_\lambda=c_\lambda I$ with $c_\lambda<0.$
To prove $\mu$ is semisimple, it suffices to show   that $\mathfrak{l}=(\lambda,\mathbb{C}^{n})$ is semisimple. Consider the following   orthogonal decompositions: (i) $\mathfrak{l}=\mathfrak{h}\oplus \mathfrak{s}$, where $\mathfrak{s}$ is the radical of $\lambda;$  (ii)  $\mathfrak{s}=\mathfrak{a}\oplus\mathfrak{n}_\lambda$,  where $\mathfrak{n}_\lambda=\lambda(\mathfrak{s},\mathfrak{s})$  is a  nilpotent ideal of $\mathfrak{l}$;  (iii) $\mathfrak{n}_\lambda=\mathfrak{v}\oplus\mathfrak{z}_\lambda$, where  $\mathfrak{z}_\lambda=\{Z\in \mathfrak{n}_\lambda:\lambda(Z,\mathfrak{n}_\lambda)=\lambda(\mathfrak{n}_\lambda,Z)=0\}$ is the center of $\mathfrak{n}_\lambda$. Clearly, $\mathfrak{z}_\lambda$ is a   ideal of $\mathfrak{l}$.
We have $\mathfrak{l}=\mathfrak{h}\oplus\mathfrak{a}\oplus\mathfrak{v}\oplus\mathfrak{z}_\lambda.$
Suppose that $\mathfrak{z}_\lambda\ne0.$ Let $\{H_i\},\{A_i\},\{V_i\},\{Z_i\}$  be an orthonormal basis of $\mathfrak{h},\mathfrak{a},\mathfrak{v},$ and $\mathfrak{z}_\lambda,$ respectively. Put $\{X_i\}=\{H_i\}\cup\{A_i\}\cup\{V_i\}\cup\{Z_i\}.$ For any $0\ne Z\in \mathfrak{z}_\lambda$, by hypothesis we have
\begin{align*}
0>\langle\textnormal{M}_{\lambda} Z, Z\rangle
=&2 \sum_{i j}|\langle\lambda(X_{i}, X_{j}), Z\rangle|^{2}-2\sum_{ij}|\langle\lambda(Z, X_{i}), X_{j}\rangle|^{2}
-2\sum_{ij}|\langle\lambda(X_{i},Z), X_{j}\rangle|^{2} \\
=&2 \sum_{ij}\left\{|\langle\lambda(Z_i, H_{j}), Z\rangle|^{2}+|\langle\lambda(H_{i},Z_j), Z\rangle|^{2}
+|\langle\lambda(Z_i, A_{j}), Z\rangle|^{2}+|\langle\lambda(A_{i},Z_j), Z\rangle|^{2}\right\}+\alpha{(Z)}\\
&-2\sum_{ij}\left\{|\langle\lambda(Z, H_{i}), Z_{j}\rangle|^{2}+|\langle\lambda(Z, A_{i}), Z_{j}\rangle|^{2}\right\}
-2\sum_{ij}\left\{|\langle\lambda(H_{i},Z), Z_{j}\rangle|^{2}+|\langle\lambda(A_{i},Z), Z_{j}\rangle|^{2}\right\},
\end{align*}
where $\alpha{(Z)}=2\sum_{ij}|\langle\lambda(Y_{i}, Y_{j}), Z\rangle|^{2}\geq 0,$ $\{Y_i\}=\{H_i\}\cup\{A_i\}\cup\{V_i\}.$ This implies
\begin{align*}
0>\sum_{k}\langle\textnormal{M}_{\lambda} Z_k, Z_k\rangle=\sum_{k}\alpha{(Z_k)}\geq 0,
\end{align*}
which is a contradiction. So $\mathfrak{z}_\lambda=0$,  and consequently, $\mathfrak{n}_\lambda=\lambda(\mathfrak{s},\mathfrak{s})=0.$

Suppose that $\mathfrak{s}\ne0.$   Let $\{H_i\},\{A_i\}$  be an orthonormal basis of $\mathfrak{h},\mathfrak{s},$  respectively. For any $0\ne A\in \mathfrak{s}$, we have
\begin{align*}
0>\langle\textnormal{M}_{\lambda} A, A\rangle
=&2\sum_{ij}\left\{|\langle\lambda(H_{i},A_j), A\rangle|^{2}+|\langle\lambda(A_{i},H_j), A\rangle|^{2}\right\}+\beta{(A)}\\
&-2\sum_{ij}|\langle\lambda(A,H_{i}), A_j\rangle|^{2}-2\sum_{ij}|\langle\lambda(H_{i},A), A_j\rangle|^{2}
\end{align*}
where $\beta{(A)}=2\sum_{ij}|\langle\lambda(H_i, H_{j}), A\rangle|^{2}\geq 0.$  This implies
\begin{align*}
0>\sum_{k}\langle\textnormal{M}_{\lambda} A_k, A_k\rangle=\sum_{k}\beta{(A_k)}\geq 0,
\end{align*}
which is a contradiction. So $\mathfrak{s}=0$. Therefore $\lambda$ is a  semisimple Lie algebra.
\end{proof}


\begin{remark}\label{negativeDefinite}
By the proof of Theorem~\ref{x}, we know that if  $[\mu]\in L_n$ for which
there exists  $[\lambda] \in \textnormal{GL}(n).[\mu]$ such that $\textnormal{M}_{\lambda}$ is  negative definite, then $\mu $ is a semisimple Lie algebra.
\end{remark}

The next  theorem shows  that   in the frame of Leibniz algebras,    the maximum value of $F_{n}: L_{n} \rightarrow \mathbb{R}$ is    attained  at  symmetric Leibniz algebras that are  non-Lie.
\begin{theorem}\label{max}
The functional $F_{n}: L_{n} \rightarrow \mathbb{R}$ attains its maximal value at a point $[\mu]\in L_n,$ $n\ge2$ if and only if $\mu$ is isomorphic to the direct sum of the two-dimensional non-Lie symmetric Leibniz algebra with the trivial algebra.   In such a case, $F_{n}([\mu])=20.$
\end{theorem}
\begin{proof}
Assume that $F_{n}: L_{n} \rightarrow \mathbb{R}$ attains its maximal value at a point $[\mu]\in L_n,$  $n\ge2.$ By Theorem~\ref{MID}, we know that $[\mu]$ is also a critical of   $F_{n}: \mathbb{P}V_{n} \rightarrow \mathbb{R}.$ Then it follows  Theorem~\ref{min_cri} that $F_{n}|_{\textnormal{GL}(n).[\mu]}$ also attains its minimum value at a point $[\mu]$ , consequently   $F_{n}|_{\textnormal{GL}.[\mu]}$ is a constant, so
\begin{align}\label{GLU}
\textnormal{GL}(n).[\mu]=\textnormal{U}(n).[\mu]
\end{align}
The  relation (\ref{GLU}) implies that  the only non-trivial degeneration  of $\mu$ is $0$ (\cite[Theorem 5.1]{Lauret2003}), consequently the  degeneration level of $\mu$ is $1$. By \cite{KO2013}, a  Leibniz algebra of  degeneration level $1$  is necessarily isomorphic to one of the following
\begin{enumerate}
\item [(i)]$\mu_{hy}$ is a Lie algebra: $\mu_{hy}(X_1,X_i)=X_i,~ i=2,\cdots,n;$
\item [(ii)]$\mu_{he}$ is a Lie algebra: $\mu_{he}(X_1,X_2)=X_3;$
\item [(iii)] $\mu_{sy}$ is a symmetric Leibniz  algebra: $\mu_{sy}(X_1,X_1)=X_2;$
\end{enumerate}
where $\{X_1,X_2,\cdots,X_n\}$ is a basis.
It is easy to see that  the critical point  $[\mu_{hy}]$ is of type $(0<1;1,n-1)$,
 $[\mu_{he}]$ is of type $(2<3<4;2,n-3,1)$ and  $[\mu_{sy}]$ is of type $(3<5<6;1,n-2,1)$. By Lemma~\ref{CV}, we know
\begin{align*}
F_{n}([\mu_{hy}])=4, \quad F_{n}([\mu_{he}])=12, \quad F_{n}([\mu_{sy}])=20.
\end{align*}
The theorem  therefore is proved.
\end{proof}

\subsection{The structure for the critical points of $F_{n}: S_{n} \rightarrow \mathbb{R}$} Note that the  maxima and minima of  the functional $F_{n}: L_{n} \rightarrow \mathbb{R}$ are actually attained at  symmetric Leibniz algebras.   In the sequel,   we characterize  the structure for the critical points of $F_{n}: S_{n} \rightarrow \mathbb{R}$ by    Theorem~\ref{eigenvalue}. These are  main results of this article.
\begin{theorem}\label{structure}
Let $[\mu]\in S_{n}$ be a critical point of $F_{n}: S_{n} \rightarrow \mathbb{R}$  with $\textnormal{M}_\mu=c_\mu I+D_\mu$ of type $(0<k_2<\cdots<k_r;d_1,d_2,\cdots,d_r)$ and consider
\begin{align}\label{ortho}
\mathfrak{l}=\mathfrak{l}_0\oplus\mathfrak{l}_{+},
\end{align}
 the direct sum of eigenspaces of $D_\mu$  with eigenvalues equal to zero,  and larger than zero, respectively. Then the following statements hold:
\begin{enumerate}
\item [(i)] $(L_A^{\mu})^{*},(R_A^{\mu})^{*}\in \textnormal{Der}(\mu)$ for any $A\in \mathfrak{l}_0.$
\item [(ii)]$\mathfrak{l}_0$ is a reductive Lie subalgebra, i.e., a direct sum of the center and a semisimple ideal.
\item [(iii)] $L_Z^{\mu},R_Z^{\mu}$ are  normal operators for any $Z\in \mathfrak{z}(\mathfrak{l}_0),$  where $\mathfrak{z}(\mathfrak{l}_0)$ denotes the center of $\mathfrak{l}_0.$
\item [(iv)] $\mathfrak{l}_{+}$ is the nilradical  of $\mu$, and  it corresponds to  a critical point of type $(k_2<\cdots<k_r;d_2,\cdots,d_r)$ for the functional  $F_{m}: S_{m} \rightarrow \mathbb{R}$, where $m=\dim \mathfrak{l}_{+}.$
\end{enumerate}
\end{theorem}
\begin{proof}
For (i), since $D_\mu, L_A^{\mu}$ and $ R_A^{\mu}$ are derivations of $\mu$, we have
\begin{align*}
[D_\mu,L_A^{\mu}]=L_{D_\mu A}^{\mu}=0,\\
[D_\mu,R_A^{\mu}]=R_{D_\mu A}^{\mu}=0,
\end{align*}
for any $A\in \mathfrak{l}_0.$ Then it follows that
\begin{align*}
\operatorname{tr} \textnormal{M}_\mu[L_A^{\mu},(L_A^{\mu})^{*}]
&=\operatorname{tr}  (c_\mu I+ D_\mu)[L_A^{\mu},(L_A^{\mu})^{*}]\\
&=\operatorname{tr}  D_\mu[L_A^{\mu},(L_A^{\mu})^{*}]\\
&=\operatorname{tr}  [D_\mu,L_A^{\mu}](L_A^{\mu})^{*}\\
&=0.
\end{align*}
So $(L_A^{\mu})^{*}\in \textnormal{Der}(\mu)$ by Corollary~\ref{MD}.  Similarly, we have $(R_A^{\mu})^{*}\in \textnormal{Der}(\mu)$. This proves (i).

For (ii),    let $\mathfrak{l}_0=\mathfrak{h}\oplus \mathfrak{z}$ be the orthogonal decomposition, where $\mathfrak{h}=\mu(\mathfrak{l}_0,\mathfrak{l}_0).$  We claim that $\mathfrak{z}$ is the center of $\mathfrak{l}_0.$  Indeed,  by  the orthogonal decomposition of eigenspaces (\ref{ortho}), we have
\begin{align*}
L_A^{\mu}=\left(\begin{array}{cc}
L_A^{\mu}|_{\mathfrak{l}_0}&0\\
0&L_A^{\mu}|_{\mathfrak{l}_+}
\end{array}\right), \quad
R_A^{\mu}=\left(\begin{array}{cc}
R_A^{\mu}|_{\mathfrak{l}_0}&0\\
0&R_A^{\mu}|_{\mathfrak{l}_+}
\end{array}\right),
\end{align*}
for any $A\in \mathfrak{l}_0.$ Since $\mathfrak{h}$ is $\textnormal{Der}(\mathfrak{l}_0)$-invariant, then by (i) we know that $L_A^{\mu}|_{\mathfrak{l}_0},R_A^{\mu}|_{\mathfrak{l}_0}\in \textnormal{Der}(\mathfrak{l}_0)$ are of the form
\begin{align*}
L_A^{\mu}|_{\mathfrak{l}_0}=\left(\begin{array}{cc}
L_A^{\mu}|_{\mathfrak{h}}&0\\
0&0
\end{array}\right), \quad
R_A^{\mu}|_{\mathfrak{l}_0}=\left(\begin{array}{cc}
R_A^{\mu}|_{\mathfrak{h}}&0\\
0&0
\end{array}\right),
\end{align*}
for any $A\in \mathfrak{l}_0.$
So $\mu(\mathfrak{l}_0,\mathfrak{z})=\mu(\mathfrak{z},\mathfrak{l}_0)=0,$ i.e.,  $\mathfrak{z}$ lies in the center of $\mathfrak{l}_0.$   Moreover,  it follows that  $\mathfrak{h}=\mu(\mathfrak{h},\mathfrak{h}).$ Let $\mathfrak{h}=\bar{\mathfrak{r}}\oplus \bar{\mathfrak{s}}$ be the orthogonal decomposition, where $\bar{\mathfrak{s}} $ is the radical of $\mathfrak{h}.$ Since $\bar{\mathfrak{s}} $ is $\textnormal{Der}(\mathfrak{h})$-invariant, then by (i), we know that $L_H^{\mu}|_{\mathfrak{h}},R_H^{\mu}|_{\mathfrak{h}}\in \textnormal{Der}(\mathfrak{h})$ are of the form
\begin{align*}
L_H^{\mu}|_{\mathfrak{h}}=\left(\begin{array}{cc}
L_H^{\mu}|_{\bar{\mathfrak{r}}}&0\\
0&L_H^{\mu}|_{\bar{\mathfrak{s}} }
\end{array}\right),\quad
R_H^{\mu}|_{\mathfrak{h}}=\left(\begin{array}{cc}
R_H^{\mu}|_{\bar{\mathfrak{r}}}&0\\
0&R_H^{\mu}|_{\bar{\mathfrak{s}} }
\end{array}\right),
\end{align*}
for any $H\in \mathfrak{h}.$ Clearly, $\bar{\mathfrak{r}}$ is an ideal of $\mathfrak{h}$, and $\mathfrak{h}=\mu(\mathfrak{h},\mathfrak{h})=\mu(\bar{\mathfrak{r}},\bar{\mathfrak{r}})\oplus\mu(\bar{\mathfrak{s}},\bar{\mathfrak{s}} )$. So $\bar{\mathfrak{s}} =\mu(\bar{\mathfrak{s}},\bar{\mathfrak{s}} ).$ Since $\bar{\mathfrak{s}} $ is solvable, we conclude that $\bar{\mathfrak{s}}=0.$ Therefore $\mathfrak{h}$ is a semisimple Lie algebra by Theorem~\ref{levi}, and moreover we deduce that  $\mathfrak{z}$ is the center of $\mathfrak{f}.$
  This proves (ii).

For (iii), assume that $Z\in\mathfrak{z}$, then  by (i) we know that  the derivations $(L_Z^{\mu})^{*},(R_Z^{\mu})^{*}$ vanish on $\mathfrak{l}_0$,  and in particularly, $(L_Z^{\mu})^{*}Z=0, (R_Z^{\mu})^{*}Z=0$. Hence
\begin{align*}
[(L_Z^{\mu})^{*},L_Z^{\mu}]=0,\quad [(R_Z^{\mu})^{*},R_Z^{\mu}]=0.
\end{align*}
That is,  $L_Z^{\mu}$ and $R_Z^{\mu}$ are normal.   This proves (iii).

For (iv),  it  follows from (ii) that $\mathfrak{s}:=\mathfrak{z}\oplus\mathfrak{l}_{+}$ is the radical of $\mathfrak{l}.$ Assume that $Z\in\mathfrak{z}$ belongs to the nilradical of $\mu$, then     $L_Z^{\mu}$ and $R_Z^{\mu}:\mathfrak{l}\rightarrow\mathfrak{l}$ are necessarily nilpotent.  Together with (iii), we see  that  $L_Z^{\mu}$ and $R_Z^{\mu}$ are both normal and nilpotent,  so   $L_Z^{\mu}=R_Z^{\mu}=0,$ i.e.,  $Z$ lies in the center of $\mathfrak{l}$. This, however, contradicts $Z\in\mathfrak{l}_0$. So $Z=0$, and $\mathfrak{l}_{+}$ is the nilradical of $\mathfrak{l}$.
Set   $\mathfrak{n}:=\mathfrak{l}_{+}$, and denote  by $\mu_{\mathfrak{n}}$   the corresponding element  in $S_m$, where $m=\dim \mathfrak{l}_{+}$.
Assume that  $\{A_i\}$  is an orthonormal basis of $\mathfrak{l}_{0}$,
then by (\ref{Mformula}), we have
\begin{align}\label{n}
{\textnormal{M}_\mu}|_{\mathfrak{n}}=\textnormal{M}_{\mu_{\mathfrak{n}}}
+2\sum_{i}([L_{A_i}^{\mu},(L_{A_i}^{\mu})^{*}]+[R_{A_i}^{\mu},(R_{A_i}^{\mu})^{*}])|_{\mathfrak{n}}.
\end{align}
Using  (i) and Corollary~\ref{MD},  it follows that
\begin{align*}
\operatorname{tr}\textnormal{M}_{\mu_{\mathfrak{n}}}[L_{A_i}^{\mu},(L_{A_i}^{\mu})^{*}]|_{\mathfrak{n}}
=\operatorname{tr}\textnormal{M}_{\mu_{\mathfrak{n}}}[R_{A_i}^{\mu},(R_{A_i}^{\mu})^{*}]|_{\mathfrak{n}}=0.
\end{align*}
Since  $\operatorname{tr}\textnormal{M}_{\mu}[L_{A_i}^{\mu},(L_{A_i}^{\mu})^{*}]
=\operatorname{tr}\textnormal{M}_{\mu}[R_{A_i}^{\mu},(R_{A_i}^{\mu})^{*}]=0,$  by (\ref{n}) we have
\begin{align*}
\operatorname{tr}\textnormal{M}_{\mu}[L_{A_i}^{\mu},(L_{A_i}^{\mu})^{*}]
&=\operatorname{tr}\textnormal{M}_{\mu}|_{\mathfrak{n}}[L_{A_i}^{\mu},(L_{A_i}^{\mu})^{*}]_{\mathfrak{n}}=0,\\
\operatorname{tr}\textnormal{M}_{\mu}[R_{A_i}^{\mu},(R_{A_i}^{\mu})^{*}]
&=\operatorname{tr}\textnormal{M}_{\mu}|_{\mathfrak{n}}[R_{A_i}^{\mu},(R_{A_i}^{\mu})^{*}]_{\mathfrak{n}}=0.
\end{align*}
Put $T=\sum_{i}([L_{A_i}^{\mu},(L_{A_i}^{\mu})^{*}]+[R_{A_i}^{\mu},(R_{A_i}^{\mu})^{*}])|_{\mathfrak{n}}$, then  we have $\operatorname{tr} T^2=0$. Noting  that $T$ is Hermitian,  we conclude  $T=0.$ So $\mathfrak{n}=\mathfrak{l}_{+}$ corresponds to  a critical point of  type $(k_2<\cdots<k_r;d_2,\cdots,d_r)$ for the functional  $F_{m}: S_{m} \rightarrow \mathbb{R}$.
\end{proof}

\begin{remark}\label{s}
 Assume  that $[\mu]\in L_{n}$ is an arbitrary  critical point of $F_{n}: L_{n} \rightarrow \mathbb{R}$, and  $R_A^{\mu}\in \textnormal{Der}(\mu)$ for any $A\in \mathfrak{n}^{\perp}$ where $\mathfrak{n}$  denotes the direct  sum of  eigenspaces of $D_\mu$  with eigenvalues  larger than zero.
Then  we obtain  the same conclusions as in Theorem~\ref{structure}, except for that the nilradcal of $\mu$  might be a non-symmetric Leibniz algebra (see  Remark~\ref{n-e}).
\end{remark}

In the sequel, we characterize the critical points that lie in $S_n$ in terms of those which are nilpotent.

\begin{theorem}[Solvable extension]\label{solvable}
Assume that $\mathfrak{a}$ is an  abelian  Lie algebra of dimension $d_1$, and  $[\lambda]$ is a critical point of $F_{m}: S_{m} \rightarrow \mathbb{R}$  of type $(k_2<\cdots<k_r;d_2,\cdots,d_r)$ where $k_2>0.$ Consider the  direct sum
\begin{align*}
\mu=\mathfrak{a}\ltimes_\rho\lambda,
\end{align*}
where $\rho=(L^{\rho},R^{\rho})$,  and  $L^{\rho}:\mathbb{C}^{d_1}\times \mathbb{C}^{m}\rightarrow\mathbb{C}^{m}$, $R^{\rho}:\mathbb{C}^{m}\times \mathbb{C}^{d_1}\rightarrow\mathbb{C}^{m}$ are bilinear mappings, such that $\mu$ is  a symmetric Leibniz algebra with bracket relations given by
$$\mu(A+X,B+Y):=L^{\rho}_A(Y)+R^{\rho}_B(X)+\lambda (X,Y)$$
for all $A,B\in\mathbb{C}^{d_1}$, $X,Y\in \mathbb{C}^{m}.$
Assume that the following  conditions are satisfied
\begin{enumerate}
\item [(i)]  $[D_\lambda,L^{\rho}_A]=0,[D_\lambda,R^{\rho}_A]=0$, $\forall A\in \mathbb{C}^{d_1}.$
\item [(ii)]  $[L^{\rho}_A,(L^{\rho}_A)^{*}]=0,[R^{\rho}_A,(R^{\rho}_A)^{*}]=0$, $\forall A\in \mathbb{C}^{d_1};$ and for each $0\ne A\in \mathbb{C}^{d_1},$ $L^{\rho}_A$ or $R^{\rho}_A$ is not zero.
\end{enumerate}
If we extend the Hermitian inner product on $\mathbb{C}^{m}$ by setting
\begin{align*}
 \langle A,B\rangle=-\frac{2}{c_\lambda}(\operatorname{tr}L^{\rho}_A(L^{\rho}_B)^{*}+\operatorname{tr}R^{\rho}_A(R^{\rho}_B)^{*}),  ~~A,B\in \mathbb{C}^{d_1},
\end{align*}
 then $[\mu]$ is a solvable critical point  of type $(0<k_2<\cdots<k_r;d_1,d_2,\cdots,d_r)$ for  $F_{n}: S_{n} \rightarrow \mathbb{R}$, $n=d_1+m.$
\end{theorem}
\begin{proof}
Put $\mathfrak{n}=(\mathbb{C}^{m},\lambda)$, and let  $\{X_i\}$ be an orthonormal basis  of $\mathbb{C}^{m}$. It follows from the condition (ii) that $(L^{\rho}_A)^{*},(R^{\rho}_A)^{*}\in\textnormal{Der}(\lambda)$ for   all $A\in \mathbb{C}^{d_1}.$
Then  we have
\begin{align*}
\langle\textnormal{M}_{\mu} X, A\rangle&=-2 \sum_{i,j}\langle\mu(X_i,X), X_{j}\rangle \overline{\langle\mu(X_i,A), X_{j}\rangle}
-2 \sum_{i,j}\langle\mu(X, X_i), X_{j}\rangle \overline{\langle\mu(A, X_i), X_{j}\rangle}\\
&=-2 \sum_{i,j}\langle\lambda(X_i,X), X_{j}\rangle \overline{\langle\mu(X_i,A), X_{j}\rangle}
-2 \sum_{i,j}\langle\lambda(X, X_i), X_{j}\rangle \overline{\langle\mu(A, X_i), X_{j}\rangle}\\
&=-2\operatorname{tr}(R^{\rho}_A)^{*}R^{\lambda}_X-2\operatorname{tr}(L^{\rho}_A)^{*}L^{\lambda}_X\\
&=0,
\end{align*}
for  any $A\in \mathbb{C}^{d_1}, X\in\mathbb{C}^{m}$  since $\lambda$ is nilpotent  and  $(L^{\rho}_A)^{*},(R^{\rho}_A)^{*}\in\textnormal{Der}(\mathfrak{\lambda}).$   So $\textnormal{M}_{\mu}$ leaves  $\mathfrak{a}$ and $\mathfrak{n}$ invariant,  and moreover, it is not hard  to see that  $\textnormal{M}_\mu|_\mathfrak{n}=\textnormal{M}_\lambda=c_\lambda I+D_\lambda$  by (\ref{Mformula}).
On the other hand, we have
\begin{align*}
\langle\textnormal{M}_{\mu} A, B\rangle&=-2 \sum_{i,j}\langle\mu(X_i,A), X_{j}\rangle \overline{\langle\mu(X_i,B), X_{j}\rangle}
-2 \sum_{i,j}\langle\mu(A, X_i), X_{j}\rangle \overline{\langle\mu(B, X_i), X_{j}\rangle}\\
&=-2(\operatorname{tr}L^{\rho}_A(L^{\rho}_B)^{*}+\operatorname{tr}R^{\rho}_A(R^{\rho}_B)^{*})\\
 &=c_\lambda \langle A, B\rangle,
\end{align*}
for any $A,B\in \mathbb{C}^{d_1}.$ So $\textnormal{M}_\mu=c_\mu I+D_\mu,$ where $c_\mu=c_\lambda$ and
\begin{align*}
D_\mu=\left( {\begin{array}{*{20}{c}}
	0&0\\
	0&D_\lambda\\
\end{array}} \right)\in\textnormal{Der}(\mu).
\end{align*}
This completes the proof.
\end{proof}

\begin{theorem}[General extension]\label{converse}
Assume that $\mathfrak{f}=\mathfrak{h}\oplus\mathfrak{z}$ is a reductive Lie algebra of dimension $d_1$, and  $[\lambda]$ is a critical point of $F_{m}: S_{m} \rightarrow \mathbb{R}$  of type $(k_2<\cdots<k_r;d_2,\cdots,d_r)$ where $k_2>0.$ Consider the   direct sum
\begin{align*}
\mu=\mathfrak{f}\ltimes_\rho\lambda,
\end{align*}
where  $\rho=(L^{\rho},R^{\rho})$, and  $L^{\rho}:\mathbb{C}^{d_1}\times \mathbb{C}^{m}\rightarrow\mathbb{C}^{m}$, $R^{\rho}:\mathbb{C}^{m}\times \mathbb{C}^{d_1}\rightarrow\mathbb{C}^{m}$ are bilinear mappings, such that $\mu$ is  a symmetric Leibniz algebra with bracket relations given by
$$\mu(A+X,B+Y):=\operatorname{ad}_{\mathfrak{f}}A(B)+L^{\rho}_A(Y)+R^{\rho}_B(X)+\lambda (X,Y)$$
for all $A,B\in \mathbb{C}^{d_1}$, $X,Y\in \mathbb{C}^{m}.$
Assume that the following  conditions are satisfied
\begin{enumerate}
\item [(i)]   $[D_\lambda,L^{\rho}_A]=0,[D_\lambda,R^{\rho}_A]=0$, $\forall A\in \mathbb{C}^{d_1}.$
\item [(ii)] $[L^{\rho}_Z,(L^{\rho}_Z)^{*}]=0,[R^{\rho}_Z,(R^{\rho}_Z)^{*}]=0,$ $\forall Z\in\mathfrak{z};$ and  for each  $0\ne Z\in \mathfrak{z},$ $L^{\rho}_Z$ or $R^{\rho}_Z$ is not zero.
\end{enumerate}
Let $\langle\cdot,\cdot\rangle_1$  be a  Hermitian inner product on $\mathfrak{f}$ and   $\{H_i~|~H_i\in \mathfrak{h}\}\cup\{Z_i~|Z_i\in\mathfrak{z}\}$ be an orthonormal basis  of $(\mathfrak{f},\langle\cdot,\cdot\rangle_1)$  such that $(\operatorname{ad}_{\mathfrak{f}}H_i)^{*1}=-\operatorname{ad}_{\mathfrak{f}}H_i$, $(L^{\rho}_{H_i})^{*}=-L^{\rho}_{H_i}$, $(R^{\rho}_{H_i})^{*}=-R^{\rho}_{H_i}$ for all $i$.
If we extend the Hermitian inner product on $\mathbb{C}^{m}$ by setting
\begin{align*}
 \langle A,B\rangle=-\frac{2}{c_\lambda}(\operatorname{tr}\operatorname{ad}_{\mathfrak{f}}A(\operatorname{ad}_{\mathfrak{f}}B)^{*1}
 +\operatorname{tr}L^{\rho}_A(L^{\rho}_B)^{*}+\operatorname{tr}R^{\rho}_A(R^{\rho}_B)^{*}),  ~~A,B\in \mathbb{C}^{d_1},
\end{align*}
 then $[\mu]$ is a critical point  of type $(0<k_2<\cdots<k_r;d_1,d_2,\cdots,d_r)$ for  $F_{n}: S_{n} \rightarrow \mathbb{R}$,  $n=d_1+m.$
\end{theorem}
\begin{proof}
Put $\mathfrak{n}=(\mathbb{C}^{m},\lambda)$, and let $\{A_i\}=\{H_i,Z_i\}$ be the orthonormal basis  of $(\mathbb{C}^{d_1},\langle\cdot,\cdot\rangle_1)$  as in hypothesis, and $\{X_i\}$ be an orthonormal basis  of $\mathbb{C}^{m}.$ Then for any $A\in \mathbb{C}^{d_1}, X\in\mathbb{C}^{m}$, we have
\begin{align*}
\langle\textnormal{M}_{\mu} X, A\rangle&=-2 \sum_{i,j}\langle\mu(X_i,X), X_{j}\rangle \overline{\langle\mu(X_i,A), X_{j}\rangle}
-2 \sum_{i,j}\langle\mu(X, X_i), X_{j}\rangle \overline{\langle\mu(A, X_i), X_{j}\rangle}\\
&=-2 \sum_{i,j}\langle\lambda(X_i,X), X_{j}\rangle \overline{\langle\mu(X_i,A), X_{j}\rangle}
-2 \sum_{i,j}\langle\lambda(X, X_i), X_{j}\rangle \overline{\langle\mu(A, X_i), X_{j}\rangle}\\
&=-2\operatorname{tr}(R^{\rho}_A)^{*}R^{\lambda}_X-2\operatorname{tr}(L^{\rho}_A)^{*}L^{\lambda}_X\\
&=0,
\end{align*}
since $\lambda$ is nilpotent  and  $(L^{\rho}_A)^{*},(R^{\rho}_A)^{*}\in\textnormal{Der}(\mathfrak{\lambda}).$ So $\textnormal{M}_{\mu}$ leaves  $\mathfrak{f}$ and $\mathfrak{n}$ invariant,  and it is not hard  to see that  $\textnormal{M}_\mu|_\mathfrak{n}=\textnormal{M}_\lambda=c_\lambda I+D_\lambda$  by (\ref{Mformula}).
Moreover, for any $A,B\in \mathbb{C}^{d_1}$, we have
\begin{align*}
\langle\textnormal{M}_{\mu} A, B\rangle&=2 \sum_{i,j} \overline{\langle\mu( A_i,A_j), A\rangle}\langle\mu(A_i,A_j), B\rangle\\
&\quad-2 \sum_{i,j}\langle\mu(A_i,A), A_{j}\rangle \overline{\langle\mu(A_i,B), A_{j}\rangle}-2 \sum_{i,j}\langle\mu(X_i,A), X_{j}\rangle \overline{\langle\mu(X_i,X), X_{j}\rangle}\\
&\quad -2 \sum_{i,j}\langle\mu(A, A_i), A_{j}\rangle \overline{\langle\mu(B, A_i), A_{j}\rangle}
-2 \sum_{i,j}\langle\mu(A, X_i), X_{j}\rangle \overline{\langle\mu(X, X_i), X_{j}\rangle}\\
&=-2(\operatorname{tr}\operatorname{ad}_{\mathfrak{f}}A(\operatorname{ad}_{\mathfrak{f}}B)^{*1}
 +\operatorname{tr}L^{\rho}_A(L^{\rho}_B)^{*}+\operatorname{tr}R^{\rho}_A(R^{\rho}_B)^{*})\\
 &=c_\lambda \langle A, B\rangle.
\end{align*}
So $\textnormal{M}_\mu=c_\mu I+D_\mu,$ where  $c_\mu=c_\lambda,$ and
\begin{align*}
D_\mu=\left( {\begin{array}{*{20}{c}}
	0&0\\
	0&D_\lambda\\
\end{array}} \right)\in\textnormal{Der}(\mu).
\end{align*}
This completes the proof.
\end{proof}

\begin{remark}\label{s-c}
The condition in Theorem~\ref{solvable} and Theorem~\ref{converse} can be relaxed as follows:  $[\lambda]$ is a critical point of $F_{m}: L_{m} \rightarrow \mathbb{R}$  of type $(k_2<\cdots<k_r;d_2,\cdots,d_r)$ where $k_2>0,$ and  the constructed algebra $\mu$ is  a  Leibniz algebra with $R^{\rho}_A\in\textnormal{Der}(\mu)$ for all $A\in \mathbb{C}^{d_1}.$
\end{remark}

\section{Examples}\label{examples}
In this section, we classify  the critical  points of the functional $F_n: S_n \rightarrow \mathbb{R}$  for $n=2$ and $3$, respectively.

\subsection{Two-dimensional case} Note that there are only two non-abelian two-dimensional symmetric Leibniz algebras up to isomorphism, which is defined by
\begin{align*}
 \textnormal{Lie:} ~&[e_1,e_2]=e_2;\\
 \textnormal{non-Lie:} ~&[e_1,e_1]=e_2.
\end{align*}
Indeed, endow the two algebras with the Hermitian inner product $\langle\cdot,\cdot\rangle$,   so  that  $\{e_1,e_{2}\}$ is an orthonormal  basis. Then it is easy  to see that the Lie algebra is   a critical point  of $F_2$ with type $(0<1;1,1)$, and the critical value is $4;$  The non-Lie  symmetric Leibniz algebra is   a critical point  of $F_2$ with type $(1<2;1,1)$, and the critical value is $20.$

\subsection{Three-dimensional case}The classification of $3$-dimensional symmetric Leibniz  algebras over $\mathbb{C}$ can be found in \cite{AO1998,CIL2012}. We classify the critical  points of the functional $F_3: S_3 \rightarrow \mathbb{R}$ as follows
\begin{align*}
&\text { TABLE I. }  \text { non-zero 3-dimensional symmetric Leibniz algebras, critical types and critical values. }\\
&\begin{array}{llllc}
\hline \hline \g & \text{Type}&  \text{Multiplication table }\quad\quad\quad\quad &  \text{Critical type}\quad\quad\quad\quad&\text{Critical value} \\
\hline
\text{L}_1 &\text{Lie}&\left\{ \begin{aligned}
&~~[e_{1}, e_{2}]=e_{3}
\end{aligned} \right.
&(1<2;2,1) & 12 \\
\text{L}_2 &\text{Lie}&\left\{ \begin{aligned}
&~~[e_{1}, e_{2}]=e_{2}
\end{aligned} \right.
&(0<1;1,2) & 4 \\
\text{L}_{3}(\alpha), \alpha\ne0&\text{Lie} &\left\{ \begin{aligned}
&~~[e_{3},e_{1} ]=e_{1},  [e_{3}, e_{2}]=\alpha e_{2},
\end{aligned} \right. &(0<1;1,2) & 4 \\
\text{L}_4&\text{Lie} & \left\{ \begin{aligned}
&~~[e_{3}, e_{1}]=e_{1}+e_{2}, [e_{3}, e_{2}]=e_{2}
\end{aligned} \right.
 &- & - \\
\text{L}_5&\text{Lie} &  \left\{\begin{aligned}
&[e_{3}, e_{1}]=2e_{1},[e_{3}, e_{2}]=-2e_{2}\\
&[e_{1}, e_{2}]=e_{3}
\end{aligned} \right.
 &  (0;3) & \frac{4}{3} \\
\text{S}_{1}&\text{non-Lie}& \left\{ \begin{aligned}
&~~[e_{3},e_{3} ]=e_{1}
\end{aligned} \right. &(3<5<6;1,1,1) & 20\\
\text{S}_{2} &\text{non-Lie}&
\left\{ \begin{aligned}
&~~[e_{2},e_{2} ]=e_{1},  [e_{3}, e_{3}]=e_{1}
\end{aligned} \right.
&(1<2;2,1) & 12\\
\text{S}_{3}(\frac{1}{4})&\text{non-Lie}&
\left\{ \begin{aligned}
	&[e_{2},e_{2} ]=\frac{1}{4}e_{1},  [e_{3}, e_{2}]=e_{1}, \\
	&[e_{3},e_{3} ]=e_{1}
\end{aligned} \right.& -  &  - \\
\text{S}_{3}(\beta),\beta\ne \frac{1}{4}&\text{non-Lie}&
\left\{ \begin{aligned}
&[e_{2},e_{2} ]=\beta e_{1},  [e_{3}, e_{2}]=e_{1}, \\
&[e_{3},e_{3} ]=e_{1}
\end{aligned} \right.&  (1<2;2,1)  &  12 \\
\text{S}_{4} &\text{non-Lie}&
\left\{ \begin{aligned}
&~~[e_{1},e_{3} ]=e_{1}
\end{aligned} \right.
&(0<1;1,2) & 4\\
\text{S}_{5}(\alpha),\alpha\ne0 &\text{non-Lie}&
\left\{ \begin{aligned}
&[e_{1},e_{3} ]=\alpha e_{1},  [e_{2}, e_{3}]=e_{2}, \\
&[e_{3},e_{2} ]=-e_{2}
\end{aligned} \right.&(0<1;1,2) & 4\\
\text{S}_{6}&\text{non-Lie} & \left\{ \begin{aligned}
&[e_{2},e_{3} ]=e_{2},  [e_{3}, e_{2}]=-e_{2}, \\
&[e_{3},e_{3} ]=e_{1}
\end{aligned} \right. &- & -\\
\text{S}_{7}(\alpha),\alpha\ne0&\text{non-Lie}&
\left\{ \begin{aligned}
&~~[e_{1},e_{3} ]=\alpha e_{1},  [e_{2}, e_{3}]=e_{2}
\end{aligned} \right. &(0<1;1,2) & 4\\
\text{S}_{8}&\text{non-Lie}&
\left\{ \begin{aligned}
&~~[e_{1},e_{3} ]=e_{1}+e_{2},  [e_{3}, e_{3}]=e_{1}
\end{aligned} \right.& - & -\\
\hline \hline
\end{array}
\end{align*}
Indeed, TABLE I are obtained from the following four steps
\begin{enumerate}
\item   For the cases $\text{L}_1,\text{S}_1,\text{S}_2$,  endow them  with the Hermitian inner product $\langle\cdot,\cdot\rangle$   so  that  $\{e_1,e_{2},e_3\}$ is an orthonormal  basis.
\item  For the cases $\text{L}_2,\text{L}_3,\text{S}_4,\text{S}_5(\alpha), \text{S}_7(\alpha)$, use Theorem~\ref{solvable}.
\item  For the cases $\text{L}_4,\text{S}_6,\text{S}_8$, use Theorem~\ref{structure}.
\item  For  $\text{S}_{3}(\beta)$, it is an associative algebra. By \cite{ZY}, we know that    $\text{S}_{3}(\frac{1}{4})$ is isomorphic to $d_{21}$,    and $\text{S}_{3}(\beta)$, $\beta\ne \frac{1}{4}$ is isomorphic to $d_{22}$.
\end{enumerate}
Together with Lemma~\ref{CV},  we complete TABLE I.

\section{Summary and Comments}\label{ques}
This article can be thought to find the 'best' Hermitian inner products   in an isomorphism class of  a given Leibniz algebra,
which are   characterized   by  the critical points of $F_n: L_n \rightarrow \mathbb{R}$. Moreover, the 'best' Hermitian inner products if exist, are unique up to scaling and isometry, and    pose a severe restriction on the algebraic  structure of the given Leibniz algebra. The main results of this article are briefly  summarized as follows
\begin{enumerate}
\item [(a)] The eigenvalue types  for the critical points of   $F_n: S_n \rightarrow \mathbb{R}$ are necessarily   nonnegative, and the nilpotent critical points of   $F_n: S_n \rightarrow \mathbb{R}$ have positive eigenvalue types (Theorem~\ref{eigenvalue}, ~\ref{positive}).
\item [(b)] The  maxima and minima of the functional  $F_n: L_n \rightarrow \mathbb{R}$  are actually attained   at the symmetric Leibniz algebras (Theorem~\ref{x}, ~\ref{max}).
\item [(c)] The  structure  of  an arbitrary  critical point  of $F_n: S_n \rightarrow \mathbb{R}$ is characterized (Theorem~\ref{structure}--\ref{converse}).
\end{enumerate}
Although some generalizations are obtained (Remark~\ref{n-e},~\ref{s},~\ref{s-c}),  we still do not have a complete understanding for the critical points of  $F_n: L_n \rightarrow \mathbb{R}$.
Based on the discussion in previous sections,  it is natural and interesting to ask the following questions.

\begin{Question}\label{nonneg}
Do all critical points of $F_n: L_n \rightarrow \mathbb{R}$ necessarily  have nonnegative eigenvalue types?
\end{Question}

\begin{Question}\label{po}
Do all nilpotent critical points of $F_n: L_n \rightarrow \mathbb{R}$ necessarily  have positive eigenvalue types?
\end{Question}

One may also ask: \textit{do all critical points of $F_n: L_n \rightarrow \mathbb{R}$ necessarily lie in $S_n$?} We point out that this does not hold,  even for $n=2.$ Consider the two-dimensional non-symmetric Leibniz  algebra $\mu$: $e_1e_2=e_2.$ Then $[\mu]$ is a critical point of $F_2: L_2 \rightarrow \mathbb{R}$ with type $(0<1;1,1).$

\section{Acknowledgement}
 This paper is partially supported by NSFC (11931009 and 12131012) and NSF of Tianjin (19JCYBJC30600).


\begin{thebibliography}{99}
\bibitem{AO1998}  S. Ayupov and  B. Omirov: {On Leibniz algebras}, in: Algebra and Operator Theory (Tashkent,1997), KluwerAcad. Publ. Dordrecht, 1998, pp.1--12.
\bibitem{Barnes2021} D. Barnes: {On Levi's theorem for Leibniz algebras}, Bull. Austral. Math. Soc. \textbf{86} (2012),
no. 2, 184--185.
\bibitem{Bloh1965}  A. Bloh: {On a generalization of the concept of Lie algebra} (in Russian), Dokl. Akad. Nauk SSSR \textbf{165} (1965), 471--473; translated into English in Soviet Math. Dokl. \textbf{6} (1965), 1450--1452.

\bibitem{BL2018} C. B$\ddot{\textnormal{o}}$hm and R. Lafuente: {Immortal homogeneous Ricci flows}, Invent. Math. \textbf{212} (2018), no. 2, 461--529.

\bibitem{BL2023} C. B$\ddot{\textnormal{o}}$hm and R. Lafuente: {Non-compact Einstein manifolds with symmetry}, to appear in  J. Amer. Math. Soc. (2023).

\bibitem {BoHo2020}  R. Bonezzi and O. Hohm: {Leibniz gauge theories and infinity structures}, Commun. Math. Phys. \textbf{377} (2020), 2027--2077.

\bibitem{CIL2012} J. Casas, M. Insua, M. Ladra and S. Ladra: {An algorithm for the classification of 3-dimensional complex Leibniz algebras}, Linear Algebra Appl. \textbf{436} (2012), no. 9, 3747--3756.

\bibitem{Feldvoss} J. Feldvoss: {Leibniz algebras as nonassociative algebras}, Vol. 721. Providence, RI: American Mathematical Society, (2019) p. 115--149.

\bibitem{GKM} C. Gorodski,  I. Kashuba and  M. Martin:    A moment map for the variety of Jordan algebras, 	arXiv:2301.10806v1.

\bibitem{HoSa2019} O. Hohm and H. Samtleben: {Leibniz-Chern-Simons theory and phases of exceptional field theory}, Commun. Math. Phys. \textbf{369} (2019), 1055--1089.

\bibitem{HTT}{K. Hiroshi, A. Takahara and H. Tamaru:} {The space of left-invariant metrics on a Lie group up to isometry and scaling}, Manuscripta Math. \textbf{135} (2011), no 1-2, 229--243.

\bibitem{KO2013}  A. Khudoyberdiyev and B. Omirov: {The classification of algebras of level one},
Linear Algebra Appl. \textbf{439}(11) (2013), 3460--3463.

\bibitem{Kirwan98} K. Kirwan: {Momentum maps and reduction in algebraic geometry}, Differ. Geom. Appl. \textbf{9} (1998)135--172.

\bibitem{KotStr} A. Kotov; T. Strobl: {The embedding tensor, Leibniz-Loday algebras, and their higher Gauge theories}, Commun. Math. Phys. \textbf{376} (2020), 235--258.

\bibitem{lauret2001} {J. Lauret:} {Ricci soliton homogeneous nilmanifolds},  Math. Ann. \textbf{{319}} (2001), 715--733.

\bibitem{Lauret03}J. Lauret: {On the moment map for the variety of Lie algebras}, J. Funct. Anal. \textbf{202} (2003), 392--423.

\bibitem{Lauret2003}J. Lauret: {Degenerations of Lie algebras and geometry of Lie groups}, Differ. Geom. Appl. \textbf{18} (2003), no. 2, 177--194.

\bibitem{Lauret2010}J. Lauret: {Einstein solvmanifolds are standard}, Ann.  Math. \textbf{172} (2010), 1859--1877.

\bibitem{Lauret2011}J. Lauret: {Ricci soliton solvmanifolds}, J. Reine. Angew. Math. \textbf{650} (2011), 1--21.

\bibitem{Lavau19} S. Lavau: {Tensor hierarchies and Leibniz algebras}, J. Geom. Phys. \textbf{144} (2019), 147--189.

\bibitem{Lavau20} S. Lavau; {Palmkvist, J.: Infinity-enhancing Leibniz algebras}, Lett. Math. Phys. \textbf{110} (2020), 3121--3152.

\bibitem{Loday1993} J.-L. Loday: {Une version non commutative des algbres de Lie: les algbres de Leibniz}. (French) [A noncommutative version of Lie algebras: the Leibniz algebras], Enseign. Math. (2) \textbf{39} (1993), no. 3-4, 269--293.

\bibitem{LodayPira1993}  J.-L. Loday and T. Pirashvili: {Universal enveloping algebras of Leibniz algebras and (co)homology}, Math. Ann. \textbf{296} (1993), no. 1, 139--158.

\bibitem{MY2013}  G. Mason and G. Yamskulna: {Leibniz algebras and Lie algebras}, SIGMA Symmetry Integrability Geom. Methods Appl. \textbf{9} (2013), Paper 063, 10 pp.

\bibitem{M76} {J. Milnor:} {Curvatures of left invariant metrics on Lie groups}, Adv. Math. \textbf{21} (1976),
293--329.

\bibitem{Ness1984} L. Ness: {A stratification of the null cone via the moment map}, {Amer. J. Math.} \textbf{106} (1984), 1281-1329 (with an appendix by D. Mumford).

\bibitem{ShTaZh2021} Y. Sheng,  R. Tang and C. Zhu: The controlling $L_\infty$-algebra, cohomology and homotopy of embedding tensors and Lie-Leibniz triples. Commun. Math. Phys. \textbf{386} (2021), 269--304.

\bibitem{Str2019} T. Strobl: {Leibniz-Yang-Mills gauge theories and the 2-Higgs mechanism}, Phys. Rev. D \textbf{99} (2019), 115026.

\bibitem{StrWa20}  T. Strobl and  F. Wagemann: {Enhanced Leibniz algebras: structure theorem and induced Lie 2-algebra},
Commun. Math. Phys. \textbf{376} (2020), 51--79.

\bibitem{Towers2021} D. Towers : {On the nilradical of a Leibniz algebra}, Commun. Algebra. \textbf{49} (2021), no. 10, 4345--4347.

\bibitem{ZCL} H. Zhang and Z. Chen;  L. Li: {The moment map for the variety of $3$-Lie algebras}, J. Funct. Anal. \textbf{283} (2022), No. 11, Article ID 109683.

\bibitem{ZY}H. Zhang and Z. Yan: The moment map for the variety of associative algebras, arXiv:2301.12142v1.

\end{thebibliography}
\end{document}